\documentclass[12pt]{article}
\usepackage[]{amsmath,amssymb}
\usepackage[]{amsthm,enumerate}
\usepackage{lscape} 

\newtheorem{theorem}{Theorem}
\newtheorem{proposition}[theorem]{Proposition}

\newtheorem{corollary}[theorem]{Corollary}
\theoremstyle{definition}

\theoremstyle{remark}
\newtheorem{remark}[theorem]{Remark}

\newcommand{\FF}{\mathbb{F}}
\DeclareMathOperator{\wt}{wt}

\begin{document}

\title{On binary codes related to mutually quasi-unbiased weighing
matrices
}

\date{\today}

\author{
Masaaki Harada\thanks{
Research Center for Pure and Applied Mathematics,
Graduate School of Information Sciences,
Tohoku University, Sendai 980--8579, Japan.
email: mharada@m.tohoku.ac.jp.}
and
Sho Suda\thanks{Department of Mathematics Education, Aichi University of Education,
1 Hirosawa, Igaya-cho, Kariya 448--8542, Japan.
email: suda@auecc.aichi-edu.ac.jp}
}

\maketitle

\vspace*{-1cm}
\begin{center}
{\bf Dedicated to Hadi Kharaghani on his 70th birthday}
\end{center}

\begin{abstract}
Some mutually quasi-unbiased weighing matrices are constructed from
binary codes satisfying 
that the number of non-zero weights of the code is four
and the code contains the first order Reed--Muller code.
Motivated by this, 
in this note, we study binary codes satisfying the conditions.
The weight distributions of binary codes satisfying the conditions are determined.
We also give a classification of binary codes of lengths $8,16$
and binary maximal codes of length $32$
satisfying the conditions.
As an application, sets of $8$ mutually quasi-unbiased 
weighing matrices for parameters $(16,16,4,64)$
and $4$ 
mutually quasi-unbiased weighing matrices for parameters $(32,32,4,256)$
are constructed for the first time.
\end{abstract}



\section{Introduction}

A weighing matrix of order $n$ and weight $k$ is an
$n \times n$ $(1,-1,0)$-matrix $W$ such that
$W W^T=kI_n$, where $I_n$ is
the identity matrix of order $n$ and $W^T$ denotes the transpose of $W$.
A weighing matrix of order $n$ and weight $n$
is a Hadamard matrix.

Two weighing matrices $W_1,W_2$ of order $n$ and weight $k$ are said to be 
{\em unbiased} if $(1/\sqrt{k})W_1 W_2^{T}$ is also 
a weighing matrix of order $n$ and weight $k$~\cite{HKO} (see also~\cite{BKR}).
Unbiased weighing matrices of order $n$ and weight $n$
are unbiased Hadamard matrices (see~\cite{HKO}). 
Weighing (Hadamard) matrices $W_1,W_2,\ldots,W_f$ are said to be 
{\em mutually unbiased}
if any distinct two of them are unbiased.
%
Generalizing the above concept, recently the concept ``quasi-unbiased'' for weighing matrices
has been introduced by Nozaki and the second author~\cite{NSpre}.
Namely,  two weighing matrices $W_1,W_2$ of order $n$ and weight $k$ are said to be  
{\em quasi-unbiased for parameters $(n,k,l,a)$} if
$(1/\sqrt{a}) W_1 W_2^T$ is a weighing matrix of weight $l$. 
It follows from the definition that $l=k^2/a$.
In addition, weighing matrices $W_1,W_2,\ldots,W_f$ are said to be 
{\em mutually quasi-unbiased weighing matrices 
for parameters $(n,k,l,a)$}  
if any distinct two of them are quasi-unbiased  for parameters $(n,k,l,a)$.
Mutually quasi-unbiased weighing matrices were defined
from the viewpoint of a connection with spherical codes~\cite{NSpre}.
This notion was introduced to show that Conjecture~23 in~\cite{BKR}
is true.
%
Only quasi-unbiased weighing matrices are previously known
for parameters $(n,n,n/2,2n)$, where $n=2^{2k+1}$ and $k$ is a positive
integer~\cite[Section~4]{BKR} and \cite[Section~4]{NSpre}, 
and for parameters $(n,2,4,1)$, where $n$ is an even positive
integer with $n \ge 4$~\cite[Section~3]{NSpre}.


Suppose that $n=2^m$, where $m$ is an integer with $m \ge 2$.
Let $C$ be a binary $[n,k]$ code satisfying the following two conditions:
\begin{align}
\label{eq:C1}
&\{i \in \{0,1,\ldots,n\}\mid A_i(C) \ne 0\}=\{0,n/2\pm a,n/2,n\}, \\
\label{eq:C2}
&\text{$C$ contains the first order Reed--Muller code $RM(1,m)$ as a
subcode},
\end{align}
where $A_i(C)$ denotes the number of codewords of weight $i$ in $C$,
and $a$ is a positive integer with $0<a<n/2$.
Then it follows from~\cite[Proposition~2.3 and Lemma~4.2]{NSpre}
that $C$ constructs a set of $2^{k-m-1}$ mutually quasi-unbiased 
weighing matrices for parameters $(n,n,(n/2a)^2,4a^2)$.

In this note, we study binary $[2^m,k]$ codes satisfying the two 
conditions~\eqref{eq:C1} and~\eqref{eq:C2}.
The weight distribution of the above code 
is determined using an integer $a$ given in~\eqref{eq:C1}.
We give a classification of binary codes $C$
satisfying the two conditions~\eqref{eq:C1} and~\eqref{eq:C2}
for lengths $8,16$.
We also give a classification of
binary maximal codes $C$ (with respect to the subspace relation)
satisfying the two conditions~\eqref{eq:C1} and~\eqref{eq:C2}
for length $32$.
As an application, sets of $8$ mutually quasi-unbiased 
weighing matrices for parameters $(16,16,4,64)$
and $4$ 
mutually quasi-unbiased weighing matrices for parameters $(32,32,4,256)$
are constructed for the first time.
All computer calculations in this note
were done by {\sc Magma}~\cite{Magma}.

\section{Mutually quasi-unbiased weighing matrices and codes} \label{sec:2}

We begin with definitions on codes used throughout this note.
A binary $[n,k]$ {\em code} $C$ is a $k$-dimensional vector subspace
of $\FF_2^n$,
where $\FF_2$ denotes the finite field of order $2$.
All codes in this note are binary.
A $k \times n$ matrix whose rows form a basis of 
$C$ is called a {\em generator matrix} of $C$.
The parameters $n$ and $k$ are called the {\em length} and the {\em dimension} of $C$,
respectively.
For a vector $x=(x_1,\ldots,x_n)$,
the set $\{i \mid x_i \ne 0\}$ is called 
the {\em support} of $x$.
The {\em weight} $\wt(x)$ of a vector $x$ is
the number of non-zero components of $x$.
The minimum non-zero weight of all codewords in $C$ is called
the {\em minimum weight} of $C$, which is denoted by $d(C)$.
Two codes $C$ and $C'$ are {\em equivalent} 
if one can be
obtained from the other by permuting the coordinates.
A code $C$ is {\em doubly even} (resp.\ {\em triply even})
if all codewords of $C$ have weight divisible by four (resp.\ eight).
The \textit{dual code} $C^{\perp}$ of a code
$C$ of length $n$ is defined as
$
C^{\perp}=
\{x \in \FF_2^n \mid x \cdot y = 0 \text{ for all } y \in C\},
$
where $x \cdot y$ is the standard inner product.
A code $C$ is called
{\em self-orthogonal} (resp.\ {\em self-dual}) if 
$C \subset C^{\perp}$ (resp.\ $C = C^{\perp}$).
A {\em covering radius} $\rho(C)$ of $C$ is 
$
\rho(C)=\max_{x\in\FF_2^n}\min_{c\in C}\wt(x-c).
$
The {\em first order Reed--Muller codes $RM(1,m)$} for all positive integer $m$  are defined recursively by
\begin{align*}
RM(1,1)&=\FF_2^2,\\
RM(1,m)&=\{(u,u),(u,u+{\bf 1}) \in \FF_2^{2^m}\mid u\in RM(1,m-1)\} \text{ for } m>1,
\end{align*}
where $\bf{1}$ denotes the all-one vector of suitable length. 


Mutually quasi-unbiased weighing matrices
for parameters $(n,n,(n/2a)^2,4a^2)$ are constructed from
$[n,k]$ codes $C$ 
satisfying the two conditions~\eqref{eq:C1} and~\eqref{eq:C2},
where $n=2^m$ and $m$ is a positive integer as 
follows~\cite[Proposition~2.3 and Lemma~4.2]{NSpre}.
Define $\psi$ as a map from $\mathbb{F}_2^n$ to 
$\{\pm1\}^n$ $(\subset \mathbb{Z}^n)$ by
 $\psi((x_i)_{i=1}^n)=(\alpha_i)_{i=1}^n$, where $\alpha_i=-1$ if $x_i=1$ and  $\alpha_i=1$ if $x_i=0$.
Note that $\wt(x-y)=j$ if and only if the standard inner product of
$\psi(x)$ and $\psi(y)$ is $n-2j$ for $x,y\in \mathbb{F}_2^n$.
Let $\{u_1,u_2,\ldots,u_{2^{k-m-1}}\}$ be a set of complete representatives of $C/RM(1,m)$. 
Since $\{i\in \{0,1,\ldots,n\} \mid A_i(RM(1,m)) \ne 0\}=\{0,n/2,n\}$,
$\psi(u_i+RM(1,m))$ is antipodal, that is, $-\psi(u_i+RM(1,m))=\psi(u_i+RM(1,m))$. 
Hence, there exists a subset $X_i$ of $\psi(u_i+RM(1,m))$ 
such that $X_i\cup(-X_i)=\psi(u_i+RM(1,m))$ and  $X_i\cap(-X_i)=\emptyset$.
For $1\leq i\leq 2^{k-m-1}$, define $H_i$ to be an $n\times n$
 $(1,-1)$-matrix whose rows consist of the vectors of $X_i$.  
Any two different vectors in $X_i$ are orthogonal 
for $1\leq i\leq 2^{k-m-1}$, which means that the matrix $H_i$ is a Hadamard matrix for $1\leq i\leq 2^{k-m-1}$.
Let $x_i$ be a vector in $X_i$.
The assumption of~\eqref{eq:C1}
implies that $\wt(\psi^{-1}(x_i)-\psi^{-1}(x_j))=n/2,n/2\pm a$ $(i \ne j)$,
namely,
the inner product of $x_i$ and $x_j$ $(i \ne j)$ is  $0,\mp 2a$ respectively, where $a$ is the integer given in~\eqref{eq:C1}.
This shows that for any distinct $i,j\in \{1,2,\ldots, 2^{k-m-1}\}$,
 $(1/2a)H_i H_j^T$ is a $(1,-1,0)$-matrix and thus it is a
 weighing matrix of weight $(n/2a)^2$. 
Therefore, Hadamard matrices $H_1,H_2,\ldots,H_{2^{k-m-1}}$ are mutually quasi-unbiased weighing matrices for parameters $(n,n,(n/2a)^2,4a^2)$.


\begin{remark}\label{rem}
Since $n/2a$ must be an integer, $a$ is a divisor of $2^{m-1}$.
\end{remark}

\begin{proposition} \label{prop:WD}
Suppose that $n=2^m$, where $m$ is an integer with $m \ge 2$.
Let $C$ be an $[n,k]$ code 
satisfying the two conditions~\eqref{eq:C1} and~\eqref{eq:C2}.
Then the weight distribution of $C$ is given by
\begin{align*}
(&A_0(C),A_{n/2-a}(C),A_{n/2}(C),A_{n/2+a}(C),A_n(C))\\&=(1,(2^{k-m-1}-1)l,2n-2+(2^{k-m-1}-1)(2n-2l),(2^{k-m-1}-1)l,1),
\end{align*}
where $l=(n/2a)^2$.
\end{proposition}
\begin{proof}
We denote the set of complete representatives of $C/RM(1,m)$ 
by $\{u_1,u_2,\ldots,u_{2^{k-m-1}}\}$ described as above, where
we assume that $u_1=\boldsymbol{0}$. 
%
In addition, we denote the mutually quasi-unbiased weighing 
matrices for parameters $(n,n,(n/2a)^2,4a^2)$ by $H_1,H_2,\ldots,H_{2^{k-m-1}}$
described as above.
Since $(1/2a)H_1 H_i^T$ is a weighing matrix of weight $l$ for
$1<i\leq 2^{k-m-1}$,  $0$ appears $n-l$ times 
in the first row of $(1/2a)H_1 H_i^T$. 
Since the first row of $H_1$ is the all-one vector, this implies that the number of codewords of weight $n/2$ in $u_i+RM(1,m)$  for $i>1$ is $2n-2l$.
Thus, $A_{n/2}(C)=2n-2+(2^{k-m-1}-1)(2n-2l)$ holds.
Since $C$ contains the all-one vector, we have the desired weight distribution.
\end{proof}

\begin{remark}
The minimum weight of $C$ determines the weight distribution
of $C$.
Indeed, the minimum weight determines $a$, and thus $l$. Since $k$ and $m$ are given, the weight distribution is determined. 
\end{remark}

\section{Codes satisfying  the conditions~\eqref{eq:C1} and~\eqref{eq:C2}}

In this section, 
we give a classification of codes $C$ of length $2^m$ 
satisfying the two conditions~\eqref{eq:C1} and~\eqref{eq:C2}
for $m=3,4$.
We also give a classification of
maximal codes $C$ of length $32$ 
satisfying the two conditions~\eqref{eq:C1} and~\eqref{eq:C2}.

\subsection{Length 8}
The case $m=3$ is somewhat trivial, but
we only give the result for the sake of completeness.
Note that $RM(1,3)$ is equivalent to the extended Hamming $[8,4,4]$ code
$e_8$.
The complete coset weight distribution of $e_8$ is listed 
in~\cite[Example~1.11.7]{Huffman-Pless}.
From~\cite[Example~1.11.7]{Huffman-Pless}, 
$RM(1,3)$ has 
seven (nontrivial) cosets of minimum weight $2$.
In addition, every $[8,5]$ code $C$ satisfying  
the conditions~\eqref{eq:C1} and~\eqref{eq:C2}
can be constructed as $\langle RM(1,3), x \rangle$, where $x$ is 
a coset leader of the seven cosets.
We verified by {\sc Magma} that
there exists a unique $[8,5]$ code $C_{8,5}$
satisfying  the conditions~\eqref{eq:C1} and~\eqref{eq:C2}.
This was done by the {\sc Magma} function {\tt IsIsomorphic}.
Similarly, we verified by {\sc Magma} that 
$C_{8,5}$ has 
three (nontrivial) cosets of minimum weight $2$, and
there exists a unique $[8,6]$ code $C_{8,6}$
satisfying  the conditions~\eqref{eq:C1} and~\eqref{eq:C2}.
It is trivial that the even weight $[8,7]$ code $C_{8,7}$ is the unique
$[8,7]$ code satisfying the conditions~\eqref{eq:C1} and~\eqref{eq:C2}.
We remark that
$\{i\in \{0,1,\ldots,8\} \mid A_i(C) \ne 0\}=\{0,4\pm 2,4,8\}$
for $C=C_{8,i}$ $(i=5,6,7)$.


\subsection{Length 16}
The next case is $m=4$.
First we fix the generator matrix of the first order Reed--Muller
$[16,5,8]$ code $RM(1,4)$ as follows:
\[
\left(
\begin{array}{c}
1 0 0 1 0 1 1 0 0 1 1 0 1 0 0 1 \\
0 1 0 1 0 1 0 1 0 1 0 1 0 1 0 1 \\
0 0 1 1 0 0 1 1 0 0 1 1 0 0 1 1 \\
0 0 0 0 1 1 1 1 0 0 0 0 1 1 1 1 \\
0 0 0 0 0 0 0 0 1 1 1 1 1 1 1 1 \\
\end{array}
\right).
\]
Every $[16,6]$ code $C$ satisfying  the conditions~\eqref{eq:C1} and~\eqref{eq:C2}
can be constructed as $\langle RM(1,4), x \rangle$, where $x$ is an element
of  a set of complete representatives of $\mathbb{F}_2^{16}/RM(1,4)$,
satisfying that $x+RM(1,4)$ has minimum weight $4$ or 
$6$ (see Remark~\ref{rem}).
In this way,
we found all $[16,6]$ code $C$ satisfying  the 
conditions~\eqref{eq:C1} and~\eqref{eq:C2}, which must be checked further for equivalences 
to complete the classification.
We verified by {\sc Magma} that any $[16,6]$ code
satisfying the conditions~\eqref{eq:C1} and~\eqref{eq:C2}
is equivalent to one of the two inequivalent codes 
$C_{16,6,1}$ and $C_{16,6,2}$.
This was done by the {\sc Magma} function {\tt IsIsomorphic}.
The minimum weights $d(C)$ and the constructions
of the two codes $C$ are listed in Table~\ref{Tab:16C}.
This table means that $C_{16,6,1}$ and $C_{16,6,2}$
can be constructed as 
$\langle RM(1,4), x_{16,6,1} \rangle$ and
$\langle RM(1,4), x_{16,6,2} \rangle$, respectively, 
where the supports of $x_{16,6,1}$ and $x_{16,6,2}$ are 
listed in Table~\ref{Tab:16V}.

\begin{table}[thb]
\caption{$[16,k]$ codes satisfying~\eqref{eq:C1} and~\eqref{eq:C2}}
\label{Tab:16C}
\begin{center}
{\small
\begin{tabular}{c|c|c|l}
\noalign{\hrule height0.8pt}
$k$ & Codes $C$ & $d(C)$ & \multicolumn{1}{c}{Vectors} \\
\hline
6&$C_{16,6,1}$   &6& $x_{16,6,1}$ \\
 &$C_{16,6,2}$   &4& $x_{16,6,2}$ \\
7&$C_{16,7,1}$   &6& $x_{16,6,1}$, $x_{16,7,1}$ \\
 &$C_{16,7,2}$   &4& $x_{16,6,2}$, $x_{16,7,2}$ \\
8&$C_{16,8,1}$ &4& $x_{16,6,2}$, $x_{16,7,2}$, $x_{16,18,1}$ \\
 &$C_{16,8,2}$ &4& $x_{16,6,2}$, $x_{16,7,2}$, $x_{16,18,2}$ \\
\noalign{\hrule height0.8pt}
\end{tabular}
}
\end{center}
\end{table}

\begin{table}[thb]
\caption{Vectors in Table~\ref{Tab:16C}}
\label{Tab:16V}
\begin{center}
{\small
\begin{tabular}{c|l|c|l}
\noalign{\hrule height0.8pt}
 & \multicolumn{1}{c|}{Supports} &
 & \multicolumn{1}{c}{Supports} \\
\hline
$x_{16,6,1}$ & $\{1,8,12,14,15,16\}$&$x_{16,7,2}$ & $\{1,8,10,15\}$\\
$x_{16,6,2}$ & $\{1,2,15,16\}$      &$x_{16,8,1}$ & $\{2,3,13,16\}$\\
$x_{16,7,1}$ & $\{1,4,5,7,9,10\}$   &$x_{16,8,2}$ & $\{4,5,12,13\}$\\    
\noalign{\hrule height0.8pt}
\end{tabular}
}
\end{center}
\end{table}

Let $D$ be a doubly even $[n,k]$ code satisfying the 
conditions~\eqref{eq:C1} and~\eqref{eq:C2}.
Every $[n,k+1]$ code $C$ satisfying the conditions~\eqref{eq:C1} 
and~\eqref{eq:C2} with $D \subset C$
can be constructed as $\langle D, x \rangle$, where $x$ is an element
of  a set of complete representatives of $D^\perp/D$,
satisfying that $0 \ne \wt(x) \in \{i\in \{0,1,\ldots,n\} \mid A_i(D) \ne 0\}$, 
since $D$ is self-orthogonal and 
$\{i\in \{0,1,\ldots,n\} \mid A_i(C) \ne 0\}
=\{i\in \{0,1,\ldots,n\} \mid A_i(D) \ne 0\}$.
This observation reduces the number of codes which need 
be checked for equivalences.
This observation is applied to  doubly even codes
$C_{16,6,2}$ and $C_{16,7,2}$.
Using an approach similar to the previous subsection along with
the above observation,
we completed the classification of codes
satisfying the conditions~\eqref{eq:C1} and~\eqref{eq:C2}
for dimensions $7$ and $8$.
In this case, the only results are listed in
Tables~\ref{Tab:16C} and~\ref{Tab:16V}.
We verified by {\sc Magma} that $C_{16,7,1}$ has covering
radius $4$.  
This was done by the {\sc Magma} function {\tt CoveringRadius}.
Thus, $C_{16,7,1}$ is a 
maximal code (with respect to the subspace relation).
Since $C_{16,8,1}$ and $C_{16,8,2}$ are doubly even self-dual codes,
there exists no $[16,k]$ code satisfying the conditions~\eqref{eq:C1} 
and~\eqref{eq:C2} for $k \ge 9$.
Therefore, we have the following:

\begin{proposition}
If there exists a $[16,k]$ code satisfying the conditions~\eqref{eq:C1} 
and~\eqref{eq:C2}, then $k \in \{6,7,8\}$.
Up to equivalence,
there exist two $[16,k]$ codes satisfying the conditions~\eqref{eq:C1} 
and~\eqref{eq:C2}
for $k=6,7,8$.
\end{proposition}

By the construction of quasi-unbiased 
weighing matrices described in Section~\ref{sec:2},
we have the following:

\begin{corollary}
There exists a set of at least $8$ mutually quasi-unbiased 
weighing matrices for parameters $(16,16,4,64)$.
\end{corollary}

A set of $4$ mutually quasi-unbiased 
weighing matrices for parameters $(16,16,16,16)$
is also constructed.
It is known that the maximum size 
among sets of mutually quasi-unbiased 
weighing matrices for the parameters
is $8$ \cite[Proposition~6]{CS73} and \cite[Theorem~5.2]{DGS2}.


\subsection{Length 32}
For the next case $m=5$, the classification of 
maximal codes satisfying the conditions~\eqref{eq:C1} 
and~\eqref{eq:C2} was done 
by a method  similar to that for the cases $(n,k)=(16,7), (16,8)$.

\begin{proposition}\label{prop:32}
If there exists a maximal $[32,k]$ code satisfying the 
conditions~\eqref{eq:C1} and~\eqref{eq:C2}, then $k \in \{9,10,11\}$.
Up to equivalence,
there exist $92$ maximal $[32,9]$ codes satisfying the 
conditions~\eqref{eq:C1} and~\eqref{eq:C2},
there exist $102$ maximal $[32,10]$ codes satisfying the 
conditions~\eqref{eq:C1} and~\eqref{eq:C2}, and
there exist two maximal $[32,11]$ codes satisfying the 
conditions~\eqref{eq:C1} and~\eqref{eq:C2}.
\end{proposition}


By the construction of quasi-unbiased 
weighing matrices described in Section~\ref{sec:2},
we have the following:

\begin{corollary}
There exists a set of at least $4$ 
mutually quasi-unbiased weighing matrices for parameters $(32,32,4,256)$.
\end{corollary}

A set of $8$ mutually quasi-unbiased weighing matrices for 
parameters $(32,32,16,64)$ is also constructed.
It is known that the maximum size 
among sets of mutually quasi-unbiased 
weighing matrices for the parameters 
is $32$~\cite[Theorems~4.1, 4.4]{NSpre}.

\begin{table}[thb]
\caption{Maximal $[32,k]$ codes satisfying~\eqref{eq:C1} and~\eqref{eq:C2}}
\label{Tab:32C}
\begin{center}
{\small
\begin{tabular}{c|l|c}
\noalign{\hrule height0.8pt}
$k$ & \multicolumn{1}{c|}{Codes $C$} & $d(C)$ \\
\hline
 9& $C_{32,9,1},\ldots,C_{32,9,91}$  &12 \\
   & $C_{32,9,92}$                              & 8 \\
10& $C_{32,10,1},\ldots,C_{32,10,101}$  &12 \\
   & $C_{32,10,102}$                              & 8 \\
11&$C_{32,11,1},C_{32,11,2}$              & 12\\
\noalign{\hrule height0.8pt}
\end{tabular}
}
\end{center}
\end{table}

We denote the $92$ inequivalent 
maximal $[32,9]$ codes  given in Proposition~\ref{prop:32}
by $C_{32,9,i}$ ($i=1,2,\ldots,92$).
We denote the $102$ inequivalent 
maximal $[32,10]$ codes given in Proposition~\ref{prop:32}
by $C_{32,10,i}$ ($i=1,2,\ldots,102$).
We denote the two inequivalent 
maximal $[32,11]$ codes given in Proposition~\ref{prop:32}
by $C_{32,11,i}$ ($i=1,2$).
The minimum weights of the codes given in 
Proposition~\ref{prop:32} are listed in Table~\ref{Tab:32C}.

\begin{table}[thb]
\caption{Maximal $[32,9]$ codes satisfying~\eqref{eq:C1} and~\eqref{eq:C2}}
\label{Tab:32C-9}
\begin{center}
{\small
\begin{tabular}{c|l|l}
\noalign{\hrule height0.8pt}
&\multicolumn{1}{c|}{Codes} & \multicolumn{1}{c}{Vectors} \\
\hline
$x_7$
&$C_{32,9,1},\ldots,C_{32,9,90}$ &$x_{32,7,1}$ \\
&$C_{32,9,91}$ &$x_{32,7,2}$ \\
&$C_{32,9,92}$ &$x_{32,7,3}$ \\
\hline
$x_8$
&$C_{32,9,1}, \ldots, C_{32,9,15}$  &$x_{32,8,1}$ \\
&$C_{32,9,16}, \ldots, C_{32,9,22}$  &$x_{32,8,2}$ \\
&$C_{32,9,23}, \ldots, C_{32,9,51}$  &$x_{32,8,3}$ \\
&$C_{32,9,52}, \ldots, C_{32,9,76}$  &$x_{32,8,4}$ \\
&$C_{32,9,77}, C_{32,9,78}, C_{32,9,79}$  &$x_{32,8,5}$ \\
&$C_{32,9,80}, C_{32,9,81}, C_{32,9,82}$  &$x_{32,8,6}$ \\
&$C_{32,9,83}, C_{32,9,84}, C_{32,9,85}$  &$x_{32,8,7}$ \\
&$C_{32,9,86}, C_{32,9,87}$  &$x_{32,8,8}$ \\
&$C_{32,9,88}$  &$x_{32,8,9}$ \\
&$C_{32,9,89}, C_{32,9,90}$  &$x_{32,8,10}$\\
&$C_{32,9,91}$  &$x_{32,8,11}$\\
&$C_{32,9,92}$  &$x_{32,8,12}$\\
\hline
$x_9$
&$C_{32,9,i}$ $(i=1,2,\ldots,92)$  &$x_{32,9,i}$ \\
\noalign{\hrule height0.8pt}
\end{tabular}
}
\end{center}
\end{table}

\begin{table}[thbp]
\caption{Maximal $[32,10]$ codes satisfying~\eqref{eq:C1} and~\eqref{eq:C2}}
\label{Tab:32C-10}
\begin{center}
{\footnotesize
\begin{tabular}{c|l|l|l|l}
\noalign{\hrule height0.8pt}
&\multicolumn{1}{c|}{Codes} & \multicolumn{1}{c|}{Vectors }
&\multicolumn{1}{c|}{Codes} & \multicolumn{1}{c}{Vectors } \\
\hline
$x_7$
&$C_{32,10,1},\ldots,C_{32,10,101}$ &$x_{32,7,1}$ &
$C_{32,10,102}$ & $x_{32,7,3}$ \\
\hline
$x_8$
&$C_{32,10,1},\ldots, C_{32,10,30} $   &$x_{32,8,1}$ &
$C_{32,10,31}, \ldots,  C_{32,10,73} $ &$x_{32,8,2}$ \\
&$C_{32,10,74}, \ldots,  C_{32,10,89} $&$x_{32,8,3}$  &
$C_{32,10,90}, \ldots,  C_{32,10,98} $ &$x_{32,8,4}$  \\
&$C_{32,10,99}, C_{32,10,100} $        &$x_{32,8,5}$  &
$C_{32,10,101}$                        &$y_{32,8,1}$  \\
&$C_{32,10,102}$                       &$x_{32,8,12}$ & \\
\hline
$x_9$
&$ C_{32,10,1}$ & $y_{32,9,1}$  &
$ C_{32,10, 2}, C_{32,10, 3} $ & $y_{32,9,2}$  \\
&$ C_{32,10, 9}, C_{32,10,10}, C_{32,10,11} $ & $y_{32,9,3}$  &
$ C_{32,10, 4} $ & $y_{32,9,4}$  \\
&$ C_{32,10, 5}, C_{32,10, 6} $ & $y_{32,9,5}$  &
$ C_{32,10, 7}, C_{32,10, 8} $ & $y_{32,9,6}$  \\
&$ C_{32,10,12}, C_{32,10,13} $ & $y_{32,9,7}$  &
$ C_{32,10,14}, C_{32,10,15}, C_{32,10,16} $ & $y_{32,9,8}$  \\
&$ C_{32,10,17}, C_{32,10,18} $ & $y_{32,9,9}$  &
$ C_{32,10,19}, C_{32,10,20} $ & $y_{32,9,10}$ \\
&$ C_{32,10,21} $ & $y_{32,9,11}$ &
$ C_{32,10,22} $ & $y_{32,9,12}$ \\
&$ C_{32,10,23}, C_{32,10,24} $ & $y_{32,9,13}$ &
$ C_{32,10,25}, C_{32,10,26} $ & $y_{32,9,14}$ \\
&$ C_{32,10,27} $ & $y_{32,9,15}$ &
$ C_{32,10,28} $ & $y_{32,9,16}$ \\
&$ C_{32,10,29} $ & $y_{32,9,17}$ &
$ C_{32,10,30} $ & $y_{32,9,18}$ \\
&$ C_{32,10,31}, C_{32,10,32} $ & $y_{32,9,19}$ &
$ C_{32,10,33} $ & $y_{32,9,20}$ \\
&$ C_{32,10,34}, \ldots, C_{32,10,37} $ & $y_{32,9,21}$ &
$ C_{32,10,38} $ & $y_{32,9,22}$ \\
&$ C_{32,10,39}, C_{32,10,40} $ & $y_{32,9,23}$ &
$ C_{32,10,41} $ & $y_{32,9,24}$ \\
&$ C_{32,10,42}, C_{32,10,43} $ & $y_{32,9,25}$ &
$ C_{32,10,44} $ & $y_{32,9,26}$ \\
&$ C_{32,10,45}, C_{32,10,46}, C_{32,10,47} $ & $y_{32,9,27}$ &
$ C_{32,10,48} $ & $y_{32,9,28}$ \\
&$ C_{32,10,49}, C_{32,10,50} $ & $y_{32,9,29}$ &
$ C_{32,10,51}, C_{32,10,52} $ & $y_{32,9,30}$ \\
&$ C_{32,10,53}, \ldots, C_{32,10,57} $ & $y_{32,9,31}$ &
$ C_{32,10,58} $ & $y_{32,9,32}$ \\
&$ C_{32,10,59}, C_{32,10,60} $ & $y_{32,9,33}$ &
$ C_{32,10,61} $ & $y_{32,9,34}$ \\
&$ C_{32,10,62}, C_{32,10,63}, C_{32,10,64} $ & $y_{32,9,35}$ &
$ C_{32,10,65} $     &$y_{32,9,36}$ \\
&$ C_{32,10,66} $     &$y_{32,9,37}$ &
$ C_{32,10,67}, C_{32,10,99} $ &$y_{32,9,38}$ \\
&$ C_{32,10,68} $     &$y_{32,9,39}$ &
$ C_{32,10,69} $     &$y_{32,9,40}$ \\
&$ C_{32,10,70} $     &$y_{32,9,41}$ &
$ C_{32,10,71}, C_{32,10,72} $ &$y_{32,9,42}$ \\
&$ C_{32,10,73} $     &$y_{32,9,43}$ &
$ C_{32,10,74} $     &$y_{32,9,44}$ \\
&$ C_{32,10,75} $     &$y_{32,9,45}$ &
$ C_{32,10,76}, C_{32,10,77} $ &$y_{32,9,46}$ \\
&$ C_{32,10,78},C_{32,10,79},C_{32,10,80}$&$y_{32,9,47}$ &
$ C_{32,10,81} $     &$y_{32,9,48}$ \\
&$ C_{32,10,82} $     &$y_{32,9,49}$ &
$ C_{32,10,83} $     &$y_{32,9,50}$ \\
&$ C_{32,10,84} $     &$y_{32,9,51}$ &
$ C_{32,10,85} $     &$y_{32,9,52}$ \\
&$ C_{32,10,86} $     &$y_{32,9,53}$ &
$ C_{32,10,87} $     &$y_{32,9,54}$ \\
&$ C_{32,10,88} $     &$y_{32,9,55}$ &
$ C_{32,10,89} $     &$y_{32,9,56}$ \\
&$ C_{32,10,90} $     &$y_{32,9,57}$ &
$ C_{32,10,91} $     &$y_{32,9,58}$ \\
&$ C_{32,10,92}, C_{32,10,93} $ &$y_{32,9,59}$ &
$ C_{32,10,94} $     &$y_{32,9,60}$ \\
&$ C_{32,10,95} $     &$y_{32,9,61}$ &
$ C_{32,10,96} $     &$y_{32,9,62}$ \\
&$ C_{32,10,97} $     &$y_{32,9,63}$ &
$ C_{32,10,98} $     &$y_{32,9,64}$ \\
&$ C_{32,10,100}$     &$y_{32,9,65}$ &
$ C_{32,10,101}$     &$y_{32,9,66}$ \\
&$ C_{32,10,102}$     &$y_{32,9,67}$ & \\
\noalign{\hrule height0.8pt}
\end{tabular}
}
\end{center}
\end{table}

To describe the codes given in Proposition~\ref{prop:32}, 
we fix the generator matrix of the first order Reed--Muller
$[32, 6, 16]$ code $RM(1,5)$ as follows:
\[
\left(
\begin{array}{c}
10010110011010010110100110010110\\
01010101010101010101010101010101\\
00110011001100110011001100110011\\
00001111000011110000111100001111\\
00000000111111110000000011111111\\
00000000000000001111111111111111
\end{array}
\right).
\]
The codes $C_{32,9,i}$ ($i=1,2,\ldots,92$)
are constructed as 
$\langle RM(1,5),x_7,x_8,x_9 \rangle$,
where Table~\ref{Tab:32C-9} indicates
$x_7,x_8,x_9$ and the supports are listed in Table~\ref{Tab:32V}.
The codes $C_{32,10,i}$ ($i=1,2,\ldots,102$)
are constructed as 
$\langle RM(1,5),x_7,x_8,x_9,x_{10} \rangle$,
where Table~\ref{Tab:32C-10} indicates
$x_7,x_8,x_9,x_{10}$ and the supports are 
listed in Table~\ref{Tab:32V}.
The codes $C_{32,11,i}$ ($i=1,2$)
are constructed as follows:
\[
\begin{split}
C_{32,11,1}=&\langle RM(1,5),x_{32,7,2}, z_{32,8,1}, z_{32,9,1}, z_{32,10,1}, z_{32,11,1}\rangle,\\  
C_{32,11,2}=&\langle RM(1,5),x_{32,7,2}, z_{32,8,1}, z_{32,9,2}, z_{32,10,2}, z_{32,11,2}\rangle, 
\end{split}
\]
where the supports of the vectors are listed in Table~\ref{Tab:32V}.

Finally, we compare our codes with some known codes and
we discuss the maximality of our codes.
It follows from the weight distributions that
$C_{32,9,92}$ (resp.\ $C_{32,10,102}$) is equivalent to 
the unique maximal triply even $[32,9]$ (resp.\ $[32,10]$) code,
which is given in~\cite[Table~2]{BM}.
It follows that $C_{32,9,92}$ and $C_{32,10,102}$ are maximal.
We verified by {\sc Magma} that 
$C_{32,9,1},C_{32,9,2},\ldots,C_{32,9,90}$ have covering radius $ \le 11$,
$C_{32,10,1},C_{32,10,2},\ldots,C_{32,10,101}$ have covering radius $10$ and
$C_{32,11,1},C_{32,11,2}$ have covering radius $8$.
This shows that these codes are maximal.
We verified by {\sc Magma} that $C_{32,11,2}$
is equivalent to the extended BCH $[32,11,12]$ code.

\section*{Postscript}
After this work, we continued the study of quasi-unbiased weighing matrices
obtained from (not necessary linear) codes in \cite{AHS}.

\section*{Acknowledgments}
The authors would like to thank Hiroshi Nozaki for helpful discussions. 
The authors would also like to thank the anonymous referees for
their valuable comments leading to several improvements of this note.  
This work is supported by JSPS KAKENHI Grant Number 23340021.

\begin{landscape}
\begin{table}[thbp]
\caption{Vectors for $m=5$}
\label{Tab:32V}
\begin{center}
{\footnotesize
\begin{tabular}{c|l|c|l}
\noalign{\hrule height0.8pt}
 & \multicolumn{1}{c|}{Supports} &
 & \multicolumn{1}{c}{Supports} \\
\hline
$x_{32,7,1}$&$\{1,3,4,6,7,9,10,16,17,18,19,32\}$ &
$x_{32,7,2}$&$\{1,2,3,5,6,9,10,16,17,18,19,32\}$ \\
$x_{32,7,3}$&$\{1,2,3,4,5,6,7,8\}$ &&\\
\hline
$x_{32,8,1}$&$\{1,2,4,6,7,12,13,16,24,29,30,32\}$ &
$x_{32,8,2}$&$\{4,5,6,7,8,9,11,16,24,28,30,32\}$ \\
$x_{32,8,3}$&$\{2,4,6,7,8,9,10,11,24,28,31,32\}$ &
$x_{32,8,4}$&$\{4,5,6,7,8,10,11,16,24,28,29,32\}$ \\
$x_{32,8,5}$&$\{1,4,5,6,8,9,10,11,24,28,30,32\}$ &
$x_{32,8,6}$&$\{1,2,4,6,8,9,10,16,24,28,29,31\}$ \\
$x_{32,8,7}$&$\{1,5,6,7,8,9,10,16,24,28,29,31\}$ &
$x_{32,8,8}$&$\{2,3,4,7,8,9,10,16,24,28,30,31\}$ \\
$x_{32,8,9}$&$\{2,3,6,7,8,11,13,16,24,28,29,31\}$ &
$x_{32,8,10}$&$\{1,6,8,9,10,16,24,28,29,30,31,32\}$ \\
$x_{32,8,11}$&$\{4,5,8,10,11,16,24,28,29,30,31,32\}$ &
$x_{32,8,12}$&$\{1,2,3,4,9,10,11,12\}$ \\
\hline
$x_{32,9,1}$&$\{3,4,5,6,7,9,10,11,12,16,20,21,24,28,30,32\}$ &
$x_{32,9,2}$&$\{3,4,5,6,8,9,12,15,17,28,31,32\}$ \\
$x_{32,9,3}$&$\{4,6,7,8,9,10,11,16,26,28,29,30\}$ &
$x_{32,9,4}$&$\{2,3,4,5,9,12,13,14,28,29,31,32\}$ \\
$x_{32,9,5}$&$\{1,4,5,7,8,9,10,16,24,26,28,30\}$ &
$x_{32,9,6}$&$\{3,4,7,8,9,10,12,16,20,21,26,28\}$ \\
$x_{32,9,7}$&$\{5,6,7,12,15,16,17,26,28,29,30,31\}$ &
$x_{32,9,8}$&$\{1,5,7,9,10,11,15,16,17,28,30,31\}$ \\
$x_{32,9,9}$&$\{2,3,5,6,9,10,12,16,28,30,31,32\}$ &
$x_{32,9,10}$&$\{3,4,6,7,8,9,10,12,24,28,30,31\}$ \\
$x_{32,9,11}$&$\{1,4,5,7,8,10,11,16,24,26,28,32\}$ &
$x_{32,9,12}$&$\{6,7,8,9,11,13,14,16,24,26,28,30\}$ \\
$x_{32,9,13}$&$\{3,4,6,7,8,9,11,16,24,26,28,30\}$ &
$x_{32,9,14}$&$\{1,2,4,8,9,11,13,14,28,29,31,32\}$ \\
$x_{32,9,15}$&$\{1,2,3,5,7,9,10,16,24,28,30,31\}$ &
$x_{32,9,16}$&$\{2,3,5,6,9,11,13,16,17,18,26,28\}$ \\
$x_{32,9,17}$&$\{2,3,4,5,6,9,10,11,12,16,20,21,24,26,28,29\}$ &
$x_{32,9,18}$&$\{1,4,6,9,10,12,13,16,24,26,29,30\}$ \\
$x_{32,9,19}$&$\{1,2,4,6,7,13,14,16,24,26,31,32\}$ &
$x_{32,9,20}$&$\{1,2,4,6,9,11,13,16,20,21,31,32\}$ \\
$x_{32,9,21}$&$\{1,7,11,12,15,16,17,24,26,28,30,31\}$ &
$x_{32,9,22}$&$\{1,5,6,10,11,14,15,16,17,29,30,32\}$ \\
$x_{32,9,23}$&$\{1,2,5,8,9,10,11,13,15,16,17,24,27,28,30,32\}$ &
$x_{32,9,24}$&$\{1,4,7,8,9,12,14,16,27,28,30,32\}$ \\
$x_{32,9,25}$&$\{6,7,11,13,15,16,17,24,27,28,29,31\}$ &
$x_{32,9,26}$&$\{1,3,7,8,9,11,13,16,27,28,29,32\}$\\
$x_{32,9,27}$&$\{2,8,9,10,12,13,15,16,17,24,28,30\}$&
$x_{32,9,28}$&$\{1,3,6,7,8,10,13,14,24,27,29,32\}$\\
$x_{32,9,29}$&$\{3,6,7,10,11,12,15,16,17,28,30,31\}$&
$x_{32,9,30}$&$\{1,3,4,7,10,11,12,15,17,24,28,30\}$\\
$x_{32,9,31}$&$\{1,4,5,6,10,11,14,16,27,28,30,32\}$&
$x_{32,9,32}$&$\{1,2,4,5,10,13,14,15,18,24,27,29\}$\\
$x_{32,9,33}$&$\{1,2,4,5,6,8,10,11,13,16,17,18,28,29,30,32\}$&
$x_{32,9,34}$&$\{1,2,4,5,6,11,14,16,24,27,28,30\}$\\
$x_{32,9,35}$&$\{3,4,5,7,8,12,13,15,17,27,28,29\}$&
$x_{32,9,36}$&$\{1,2,3,4,8,10,13,14,24,27,29,31\}$\\
$x_{32,9,37}$&$\{1,2,3,6,7,8,10,15,18,24,27,29\}$&
$x_{32,9,38}$&$\{3,4,5,7,10,11,14,16,27,28,29,32\}$\\
$x_{32,9,39}$&$\{1,3,4,7,9,13,14,15,17,24,27,30\}$&
$x_{32,9,40}$&$\{1,6,8,10,11,13,15,16,17,30,31,32\}$\\
$x_{32,9,41}$&$\{2,5,6,8,9,10,12,13,15,16,17,24,27,29,31,32\}$&
$x_{32,9,42}$&$\{1,2,3,4,5,8,9,10,11,13,17,18,28,29,30,32\}$\\
$x_{32,9,43}$&$\{2,3,5,6,9,11,12,15,17,24,29,31\}$&
$x_{32,9,44}$&$\{2,4,5,6,7,8,9,13,14,16,17,18,28,29,30,32\}$\\
$x_{32,9,45}$&$\{1,4,5,7,8,9,13,14,15,16,17,27,28,29,30,31\}$&
$x_{32,9,46}$&$\{4,5,7,8,10,11,13,16,17,18,28,29\}$\\
$x_{32,9,47}$&$\{1,4,5,6,7,8,9,11,14,16,17,18,27,28,30,31\}$&
$x_{32,9,48}$&$\{1,3,5,6,7,8,9,15,17,24,30,31\}$\\
\noalign{\hrule height0.8pt}
\end{tabular}
}
\end{center}
\end{table}

\setcounter{table}{5}
\begin{table}[thbp]
\caption{Vectors for $m=5$ (continued)}
\begin{center}
{\footnotesize
\begin{tabular}{c|l|c|l}
\noalign{\hrule height0.8pt}
 & \multicolumn{1}{c|}{Supports} &
 & \multicolumn{1}{c}{Supports} \\
\hline
$x_{32,9,49}$&$\{2,3,5,6,7,12,13,15,17,27,28,29\}$&
$x_{32,9,50}$&$\{2,4,5,6,10,12,13,16,27,28,30,31\}$\\
$x_{32,9,51}$&$\{2,3,4,5,6,11,14,15,17,28,30,31\}$&
$x_{32,9,52}$&$\{1,3,4,7,10,11,13,15,17,24,29,30\}$\\
$x_{32,9,53}$&$\{3,5,7,10,13,14,17,18,24,29,30,32\}$&
$x_{32,9,54}$&$\{1,2,4,5,10,13,14,15,18,24,25,31\}$\\
$x_{32,9,55}$&$\{1,3,7,9,10,11,14,16,21,22,24,28\}$&
$x_{32,9,56}$&$\{3,5,6,7,11,12,21,22,25,29,31,32\}$\\
$x_{32,9,57}$&$\{2,3,4,5,6,8,9,12,14,16,17,18,28,29,31,32\}$&
$x_{32,9,58}$&$\{2,3,4,8,9,10,12,14,17,18,29,30\}$\\
$x_{32,9,59}$&$\{1,3,5,7,8,12,13,15,17,25,28,30\}$&
$x_{32,9,60}$&$\{1,4,5,8,10,12,13,15,18,24,28,30\}$\\
$x_{32,9,61}$&$\{1,2,6,8,9,15,17,21,22,24,25,28\}$&
$x_{32,9,62}$&$\{4,6,7,8,12,13,17,18,25,28,30,32\}$\\
$x_{32,9,63}$&$\{1,4,5,6,9,10,11,14,25,30,31,32\}$&
$x_{32,9,64}$&$\{1,3,4,6,8,9,10,11,14,15,17,28,29,30,31,32\}$\\
$x_{32,9,65}$&$\{1,2,5,7,9,10,12,14,17,18,28,31\}$& 
$x_{32,9,66}$&$\{2,4,7,8,9,15,17,24,25,28,29,30\}$\\
$x_{32,9,67}$&$\{1,2,3,5,7,8,10,11,14,16,17,18,25,29,30,32\}$&
$x_{32,9,68}$&$\{1,2,5,6,7,9,11,14,24,25,28,30\}$\\
$x_{32,9,69}$&$\{6,9,10,12,14,16,21,22,24,25,30,31\}$&
$x_{32,9,70}$&$\{2,3,4,5,6,7,15,16,18,21,22,24,25,28,29,31\}$\\
$x_{32,9,71}$&$\{2,3,4,5,6,8,11,12,13,14,15,16,18,21,22,24,25,28,29,32\}$&
$x_{32,9,72}$&$\{1,2,3,8,9,11,13,16,21,22,25,29\}$\\
$x_{32,9,73}$&$\{2,3,4,5,6,8,10,11,13,15,17,21,22,24,31,32\}$&
$x_{32,9,74}$&$\{1,3,7,8,11,14,15,16,18,24,25,28,29,30,31,32\}$\\
$x_{32,9,75}$&$\{3,4,5,8,10,11,12,13,14,15,17,24,25,28,29,30\}$&
$x_{32,9,76}$&$\{2,4,5,7,8,11,13,15,17,21,22,25,28,29,31,32\}$\\
$x_{32,9,77}$&$\{3,6,7,9,10,11,12,15,17,26,28,30\}$&
$x_{32,9,78}$&$\{2,6,8,11,13,16,24,26,28,29,30,31\}$\\
$x_{32,9,79}$&$\{1,3,7,9,10,11,12,15,17,28,29,30\}$&
$x_{32,9,80}$&$\{1,3,4,8,9,10,12,15,18,24,26,30\}$\\
$x_{32,9,81}$&$\{2,3,8,9,10,12,15,16,18,29,30,32\}$&
$x_{32,9,82}$&$\{1,5,6,8,10,12,13,14,28,29,30,32\}$\\
$x_{32,9,83}$&$\{1,3,5,7,10,11,12,13,17,18,28,31\}$&
$x_{32,9,84}$&$\{1,3,4,6,8,11,13,16,24,26,28,32\}$\\
$x_{32,9,85}$&$\{1,3,4,8,11,13,15,16,17,24,26,32\}$&
$x_{32,9,86}$&$\{2,3,5,7,9,11,13,16,17,18,29,30\}$\\
$x_{32,9,87}$&$\{2,5,6,7,8,10,11,12,13,15,17,21,22,28,29,30\}$&
$x_{32,9,88}$&$\{2,4,5,6,9,12,13,14,15,16,17,24,26,29,30,32\}$\\
$x_{32,9,89}$&$\{1,4,5,7,8,9,12,14,24,27,28,31\}$&
$x_{32,9,90}$&$\{1,3,5,10,13,16,20,21,24,27,28,30\}$\\
$x_{32,9,91}$&$\{2,3,5,8,9,13,14,16,27,30,31,32\}$&
$x_{32,9,92}$&$\{1,2,3,4,17,18,19,20\}$\\
\hline
$y_{32,8,1}$&$\{1,3,4,7,11,13,15,16,17,24,28,29\}$ &&\\
\hline
$y_{32,9,1}$&$\{2,4,5,6,8,9,15,16,17,29,31,32\}$&
$y_{32,9,2}$&$\{1,2,3,4,5,9,13,14,24,26,30,31\}$\\
$y_{32,9,3}$&$\{3,4,5,6,7,11,12,15,17,28,30,32\}$&
$y_{32,9,4}$&$\{2,3,6,8,9,12,15,16,17,24,28,32\}$\\
$y_{32,9,5}$&$\{1,5,6,7,9,10,11,16,26,28,29,30\}$&
$y_{32,9,6}$&$\{1,4,5,6,8,9,10,11,15,16,17,26,28,29,31,32\}$\\
$y_{32,9,7}$&$\{4,5,6,7,8,11,12,15,17,26,28,31\}$&
$y_{32,9,8}$&$\{2,3,5,7,9,10,12,16,26,29,31,32\}$\\
$y_{32,9,9}$&$\{1,2,3,6,10,11,15,16,17,24,30,31\}$&
$y_{32,9,10}$&$\{1,3,6,9,10,11,12,16,24,26,28,30\}$\\
$y_{32,9,11}$&$\{2,7,8,9,10,11,12,16,24,28,30,31\}$&
$y_{32,9,12}$&$\{2,5,7,9,10,12,24,26,28,30,31,32\}$\\
$y_{32,9,13}$&$\{3,5,6,7,9,10,11,12,15,16,17,24,26,28,31,32\}$&
$y_{32,9,14}$&$\{1,2,3,4,8,9,13,14,24,28,29,31\}$\\
$y_{32,9,15}$&$\{2,5,6,8,11,13,14,16,26,28,29,30\}$&
$y_{32,9,16}$&$\{2,4,5,7,8,9,10,12,24,26,30,31\}$\\
\noalign{\hrule height0.8pt}
\end{tabular}
}
\end{center}
\end{table}

\setcounter{table}{5}
\begin{table}[thbp]
\caption{Vectors for $m=5$ (continued)}
\begin{center}
{\footnotesize
\begin{tabular}{c|l|c|l}
\noalign{\hrule height0.8pt}
 & \multicolumn{1}{c|}{Supports} &
 & \multicolumn{1}{c}{Supports} \\
\hline
$y_{32,9,17}$&$\{1,4,5,10,13,14,17,18,24,28,30,32\}$&
$y_{32,9,18}$&$\{2,6,8,9,10,11,12,16,24,26,29,31\}$\\
$y_{32,9,19}$&$\{3,4,6,8,9,10,11,13,26,29,31,32\}$&
$y_{32,9,20}$&$\{2,4,6,8,9,10,11,12,26,28,29,31\}$\\
$y_{32,9,21}$&$\{2,3,6,7,8,10,11,12,24,26,29,30\}$&
$y_{32,9,22}$&$\{1,7,8,10,11,13,24,26,29,30,31,32\}$\\
$y_{32,9,23}$&$\{3,5,7,8,9,10,11,13,26,28,29,31\}$&
$y_{32,9,24}$&$\{1,2,3,6,8,10,11,13,20,21,24,31\}$\\
$y_{32,9,25}$&$\{1,2,3,4,7,8,13,14,20,21,28,29\}$&
$y_{32,9,26}$&$\{2,3,5,7,8,9,11,12,13,14,24,26,29,30,31,32\}$\\
$y_{32,9,27}$&$\{1,6,8,10,11,13,20,21,24,26,28,32\}$&
$y_{32,9,28}$&$\{1,2,4,8,9,10,11,12,15,16,17,24,26,28,29,30\}$\\
$y_{32,9,29}$&$\{1,2,5,6,8,12,13,16,24,26,31,32\}$&
$y_{32,9,30}$&$\{1,2,3,4,5,7,9,16,20,21,26,28\}$\\
$y_{32,9,31}$&$\{1,2,4,7,10,12,14,16,20,21,29,32\}$&
$y_{32,9,32}$&$\{2,3,4,5,7,10,11,13,24,26,30,31\}$\\
$y_{32,9,33}$&$\{3,4,5,7,10,15,17,24,26,28,31,32\}$&
$y_{32,9,34}$&$\{1,2,3,5,6,8,15,16,17,24,29,32\}$\\
$y_{32,9,35}$&$\{1,3,5,6,7,8,9,10,20,21,26,30\}$&
$y_{32,9,36}$&$\{2,3,4,7,8,9,10,11,12,16,20,21,24,26,28,29\}$\\
$y_{32,9,37}$&$\{1,2,4,6,9,10,11,12,15,16,17,24,26,28,29,32\}$&
$y_{32,9,38}$&$\{1,3,5,6,15,16,17,24,26,30,31,32\}$\\
$y_{32,9,39}$&$\{1,2,3,4,8,10,20,21,24,26,28,29\}$&
$y_{32,9,40}$&$\{1,2,3,5,7,8,11,12,15,16,17,24,26,28,29,32\}$\\
$y_{32,9,41}$&$\{3,6,7,8,13,14,20,21,26,28,30,32\}$&
$y_{32,9,42}$&$\{1,5,7,8,10,13,14,16,20,21,26,29\}$\\
$y_{32,9,43}$&$\{1,2,3,5,6,7,13,14,20,21,26,28\}$&
$y_{32,9,44}$&$\{1,4,5,7,9,10,11,13,15,16,17,24,29,30,31,32\}$\\
$y_{32,9,45}$&$\{1,5,6,7,9,11,13,16,27,30,31,32\}$&
$y_{32,9,46}$&$\{1,2,4,5,6,13,14,15,18,27,28,29\}$\\
$y_{32,9,47}$&$\{1,6,8,9,15,16,17,28,29,30,31,32\}$&
$y_{32,9,48}$&$\{2,3,5,6,10,11,12,15,17,24,30,31\}$\\
$y_{32,9,49}$&$\{1,2,3,4,6,7,11,13,15,16,17,24,27,28,30,32\}$&
$y_{32,9,50}$&$\{2,6,9,12,13,15,17,24,27,28,29,32\}$\\
$y_{32,9,51}$&$\{5,6,7,10,11,12,13,14,24,27,29,32\}$&
$y_{32,9,52}$&$\{1,4,5,6,7,8,9,15,18,20,21,24,28,29,30,31\}$\\
$y_{32,9,53}$&$\{2,9,10,11,14,16,17,18,24,29,30,32\}$&
$y_{32,9,54}$&$\{2,4,6,8,9,12,13,16,17,18,27,28\}$\\
$y_{32,9,55}$&$\{5,6,7,8,10,11,13,16,27,28,29,30\}$&
$y_{32,9,56}$&$\{2,7,15,16,18,20,21,24,28,30,31,32\}$\\
$y_{32,9,57}$&$\{4,5,6,9,11,12,15,16,17,25,30,32\}$&
$y_{32,9,58}$&$\{1,4,5,7,8,9,10,15,18,28,30,31\}$\\
$y_{32,9,59}$&$\{2,3,6,8,9,10,11,12,25,28,29,31\}$&
$y_{32,9,60}$&$\{1,6,9,10,12,14,15,16,18,21,22,24,28,29,30,31\}$\\
$y_{32,9,61}$&$\{1,2,3,4,6,9,11,13,15,16,17,21,22,25,29,32\}$&
$y_{32,9,62}$&$\{3,5,6,7,8,12,14,16,17,18,24,30\}$\\
$y_{32,9,63}$&$\{1,2,5,7,9,12,14,16,17,18,25,28\}$&
$y_{32,9,64}$&$\{3,4,9,10,11,12,13,14,17,18,29,30\}$\\
$y_{32,9,65}$&$\{1,4,5,6,7,12,13,15,17,26,28,32\}$&
$y_{32,9,66}$&$\{2,3,4,5,6,9,11,16,24,25,29,32\}$\\
$y_{32,9,67}$&$\{1,4,6,7,10,11,13,16\}$& & \\
\hline
$z_{32,8,1}$&$\{3,4,8,12,13,14,24,28,29,30,31,32\}$&
$z_{32,9,1}$&$\{1,3,6,7,10,12,13,16,27,28,31,32\}$\\
$z_{32,10,1}$&$\{1,3,4,5,8,9,10,16,22,28,29,31\}$&
$z_{32,11,1}$&$\{1,8,9,11,12,15,17,22,28,29,30,32\}$\\
\hline
$z_{32,8,2}$&$\{3,4,8,12,13,14,24,28,29,30,31,32\}$&
$z_{32,9,2}$&$\{1,3,6,7,10,12,13,16,27,28,31,32\}$\\
$z_{32,10,2}$&$\{1,3,4,5,8,9,10,16,22,28,29,31\}$&
$z_{32,11,2}$&$\{6,7,8,9,11,12,15,16,17,27,28,30\}$\\
\noalign{\hrule height0.8pt}
\end{tabular}
}
\end{center}
\end{table}

\end{landscape}


\begin{thebibliography}{99}
\bibitem{AHS}
M. Araya, M. Harada and S. Suda,
Quasi-unbiased Hadamard matrices and weakly unbiased Hadamard matrices: a
coding-theoretic approach,
{\sl Math.\ Comp.},
(to appear), arXiv: 1504.01236.

\bibitem{BKR}
D. Best, H. Kharaghani and H. Ramp, 
Mutually unbiased weighing matrices,
{\sl Des.\ Codes Cryptogr.} {\bf 76} (2015), 237--256.

\bibitem{BM}
K. Betsumiya and A. Munemasa, 
On triply even binary codes,
{\sl J. Lond.\ Math.\ Soc.\ (2)} 
{\bf 86}  (2012),  1--16. 

\bibitem{Magma}W. Bosma, J. Cannon and C. Playoust, 
The Magma algebra system I: The user language, 
{\sl J. Symbolic Comput.}
{\bf 24} (1997), 235--265.

\bibitem{CS73}P.J. Cameron and J.J. Seidel, 
Quadratic forms over $GF(2)$, 
{\sl Nederl.\ Akad.\ Wetensch.\ Proc.\ Ser.~A}
{\bf 76}$=${\sl Indag.\ Math.}  {\bf 35}  (1973), 1--8. 


\bibitem{DGS2}P. Delsarte, J.M. Goethals and J.J. Seidel, 
Bounds for systems of lines and Jacobi polynomials, 
{\sl Philips Res.\ Rep.} {\bf 30} (1975), 91--105.


\bibitem{HKO}
W.H. Holzmann, H. Kharaghani and W. Orrick, 
On the real unbiased Hadamard matrices, 
{\sl Combinatorics and graphs}, 243--250, 
Contemp.\ Math., 531, Amer.\ Math.\ Soc., Providence, RI, 2010.

\bibitem{Huffman-Pless} W.C. Huffman and V. Pless, 
{\sl Fundamentals of Error-Correcting Codes}, 
Cambridge University Press, Cambridge, 2003. 



\bibitem{NSpre}
H. Nozaki and S. Suda, 
Weighing matrices and spherical codes,
{\sl J. Alg.\ Combin.} {\bf 42} (2015), 283--291.




\end{thebibliography}
\end{document}